\documentclass[10pt]{amsart}
\usepackage{amssymb, amscd, amsmath, amsthm, epsf, epsfig, latexsym}
\def\zz{{\bf Z}}
\def\ff{{\bf F}}
\def\qq{{\bf Q}}

\def\rr{{\bf R}}

\def\co{\colon\thinspace}
\def\calc{\mathcal{C}}

\def\calg{\mathcal{G}}

\def\calg{\mathcal{G}}
\def\calf{\mathcal{F}}

\def\co{\colon}

\newtheorem{theorem}{Theorem}[section]
\newtheorem{lemma}[theorem]{Lemma}
\newtheorem{corollary}[theorem]{Corollary}

\theoremstyle{definition}
\newtheorem{definition}[theorem]{Definition}

\theoremstyle{definition}
\newtheorem{example}[theorem]{Example}

\theoremstyle{definition}
\newtheorem{remark}[theorem]{Remark}

\numberwithin{equation}{section}

\begin{document}
\title{The  Concordance Genus of a Knot, II}
\author{Charles Livingston}
\thanks{This work was supported by a grant  from the NSF}
\address{Department of Mathematics, Indiana University, Bloomington, IN 47405}
\email{livingst@indiana.edu}
\keywords{}

\subjclass{57M25}

\maketitle


 Two basic notions of genus for knots $K \subset S^3$ are the {\it 3--genus}, $g_3(K)$, the minimum genus of an embedded surface bounded  by $K$ in $S^3$, and the {\it 4--genus}, $g_4(K)$, the minimum genus of an embedded surface bounded by $K$ in $B^4$.  A third notion is the {\it concordance genus}, $g_c(K)$, the  minimum value of $g_3(J)$ among all knots $J$ concordant to $K$.  Each of these can be defined in either the  smooth or the  topological, locally flat, category; our results apply in both.
An elementary exercise shows:
 
 \vskip.1in
 \noindent{\bf Proposition 1.} {\it  For all knots $K$, $g_3(K) \ge g_c(K) \ge g_4(K)$.  If $g_4(K) = 0$ then $g_c(K) = 0$.}
 
 \vskip.1in

There exist knots for which the gap between $g_3(K)$ and $g_c(K)$ is arbitrarily large; forming the connected sum with a slice knot does not change the value of $g_c$ but raises the values of $g_3$.  Gordon~\cite{go1} asked whether $g_c(K) = g_4(K)$ for all knots $K$.  
In unpublished work, Casson showed that for the knot $K = 6_2$, $g_c(K) = 2$ and $g_4(K) = 1$.  Nakanishi~\cite{na} proved that $g_c(K) - g_4(K)$ can be arbitrarily large.  (We use the classical names for knots, such as $6_2$, as listed in~\cite{lc, ro}.)

The article~\cite{liv1} initiated a detailed examination of $g_c$,   illustrating the use of  algebraic concordance invariants to determine the concordance genus of knots, and also demonstrating the application of Casson-Gordon invariants when algebraic invariants do not suffice.  In that article  the concordance genus was determined for all prime knots of nine or fewer crossings, excluding $8_{18}$ and $9_{40}$.  At 10 crossings the only example that does not fall to the techniques of~\cite{liv1} is $10_{82}$.   Here we delve deeper into the structure of the algebraic concordance group to prove:

\vskip.1in
 \noindent {\bf Proposition 2.} {\it For $K = 8_{18}$ and $K = 9_{40}$,  $g_3(K) = g_c(K) =3$ and $g_4(K) = 1$.}

\vskip.1in

The knot $10_{82}$ is much more interesting.  Applying Levine's classification of the algebraic concordance group, additional results about the integral algebraic concordance group, Casson-Gordon invariants, and recent work on twisted Alexander polynomials, we have:\vskip.1in

 \noindent {\bf Proposition 3.} {\it For $K = 10_{82}$, $g_3(K) = 4$ and $g_4(K) = 1$.  There are knots $J$ with $g_3(J) \le 3$ that are algebraically concordant to $K$, the first of which is $9_{42}$, but $K$ is not  concordant to any such $J$.  In particular, $g_c(10_{82}) = 4$.}

 \vskip.1in
 
\section{Basic polynomial and signature obstructions}

We begin with the computation of the concordance genus for two examples, $6_2$ and $6_2\ \# \ 6_2$ to illustrate the use of the Alexander polynomial and knot signature functions.  The first, $6_2$, was Casson's example answering Gordon's question to the negative.  To start,  we define the  normalized form of the Alexander polynomial and state the Fox-Milnor theorem on Alexander polynomials of slice knots.

\begin{definition} For a knot $K$ with Seifert form $V_K$ we let $\Delta_K(t) = t^{-a} \det(V_K - tV_K^t)$, where $a$ is chosen so that $\Delta_K(t) \in \zz[t]$ and $\Delta_K(0) \ne 0$. This is well-defined up to sign.
\end{definition}

\begin{theorem}[Fox-Milnor~\cite{fm}] If $K$ is a slice knot, then $\Delta_K(t) = t^d f(t)f(t^{-1})$ for some polynomial $f(t)$ of degree $d$.
\end{theorem}

If $K$ bounds a surface of genus $g$ in $S^3$, then it has a $2g \times 2g$ Seifert matrix, from which follows the well-known bound on the 3--genus:

\begin{theorem} For a knot $K\subset S^3$, $2 g_3(K) \ge \deg(\Delta_K(t))$.
\end{theorem}
  
\begin{example}  {\boldmath {\bf If} $K = 6_2$:  $g_3(K) = 2, g_c(K) = 2, g_4(K) =1$.}

   A Seifert surface for  $6_2$  of   genus 2 is easily found, and since  $\Delta_{6_2}(t) = 1 - 3t+3t^2-3t^3+t^4$, we have $g_3(6_2) = 2$.

Since $K$ has unknotting number one, $g_4(6_2) \le 1$; by the Fox-Milnor theorem, since $\Delta_{6_{2}}$ is irreducible, $6_2$ is not slice, so $g_4(6_2) = 1$.

Finally, one sees that $g_c(6_2) = 2$ as follows.  If $6_2$ is concordant to $J$ with $g_3(J) \le 1$, then $6_2 \# - J$ is slice and by the Fox-Milnor theorem, $\Delta_{6_2}(t)\Delta_J(t) = t^d f(t)f(t^{-1})$ for some $f(t)$.  But since $\Delta_{6_2}(t)$ is irreducible and of degree 4, while  $\deg(\Delta_J(t)) \le 2$,  this is not possible.

\end{example}

\begin{example}{\boldmath {\bf If} $K = 6_2 \# 6_2$:  $g_3(K) = 4, g_c(K) = 4, g_4(K) =2$.}

 The knot $  6_2 \# 6_2$ cannot be handled in the same way, since its Alexander polynomial is $\Delta_{6_2 \# 6_2}(t) = (1 - 3t+3t^2-3t^3+t^4)^2$, which is, in fact, the Alexander  polynomial of the slice knot  $6_2 \# {-}6_2$. By the additivity of the 3--genus, we do have that $g_3(6_2 \# 6_2) = 4$.  Introducing the signature function permits the further analysis of this example.

For any knot $K$, the  Murasugi~\cite{mu}    4--genus bound is  given by $2g_4(K) \ge |\sigma(K)|$, where $\sigma(K)$ is the signature of the symmetrized Seifert form $V_K + V_K^t$.  From our observation that $g_4(6_2) =1$ we have $g_4(6_2 \# 6_2) \le 2$; 
 also,    $\sigma(6_2 \# 6_2) = 4$, and so $g_4(6_2 \# 6_2) = 2$.
 
 The Levine-Tristram signature function of a knot, $\sigma_K(\omega)$, is   the function defined on the unit complex circle as  the local average of  the signature of the hermetianized Seifert form $(1-\omega V) + (1-\overline{\omega})V^t$, $\omega \in S^1$. (See~\cite{le2, tr} or see~\cite{go2} for a general survey of signature invariants.)   The Murasugi bound generalizes to $2g_4(K) \ge |\sigma_K(\omega)|$, and as a consequence,   $\sigma_K$ is a concordance invariant.  
 
 For a knot $K$, its signature function $\sigma_K(\omega),  \omega \in S^1$, is an integer-valued function.  The only discontinuities of  $\sigma_K(\omega) $ occur at roots of $\Delta_K(t)$.  For $\omega$ near 1, $\sigma_K(\omega) = 0$.  Thus, since $\sigma(6_2 \# 6_2) = \sigma_{6_2 \# 6_2}(-1) = 4$, we see that $\Delta_{6_2 \# 6_2}(t)$ must have a root on the unit circle and the signature function has   a jump at  one such root.  (In fact, this polynomial has a unique conjugate pair of unit roots.)
 
 If $6_2 \# 6_2$ is concordant to $J$, then the signature function $\sigma_J(\omega)$ must similarly  have a jump at a root of $\Delta_{6_2 \# 6_2}$, and it immediately follows that $1-3t+3t^2-3t^3+t^4$ divides $\Delta_J(t)$.  It then follows from  the Fox-Milnor theorem that $(1-3t+3t^2-3t^3+t^4)^2$ divides $\Delta_J(t)$, and so we see that $g_c(6_2 \# 6_2) = 4$.
 \end{example}
 
 \section{The algebraic concordance group}
 Knot signatures and exponents of symmetric irreducible factors of the Alexander polynomial yield invariants of Levine's algebraic concordance group $\calg^\zz$.  These invariants are in fact   invariants of the real algebraic concordance group, $\calg_\rr$. In~\cite{liv1} a careful study of such invariants arising from $\calg_\rr$, generalizing the examples of the previous section, was applied to determine the concordance genus of most low crossing number knots.  To extend that study we need to consider invariants of the rational algebraic concordance group, $\calg_\qq$.  We begin by reviewing some of the basic definitions and results, taken from~\cite{le2}.  
 
 The {\it algebraic concordance group} $\calg^\zz$  is defined via Seifert matrices $V$, integer matrices satisfying $\det(V - V^t ) = 1$.  Such a matrix of size $2g \times 2g$ is called {\it Witt trivial} if there is a subspace of $\qq^{2g}$ of dimension $g$ on which the bilinear form determined by $V$ is identically 0.  Two Seifert matrices $V$ and $W$ are called {\it algebraically concordant}  if $V \oplus -W$ is Witt trivial.  This is an equivalence relation and the set of equivalence classes forms the abelian group $\calg^\zz$ with operation induced by direct sum.   One of Levine's theorems is the following.
 
 \begin{theorem} If $K$ and $J$ are concordant, then $[V_K] = [V_J] \in \calg^\zz$.
 \end{theorem}
 
 Levine showed that there is an injection   $\psi\co \calg^\zz \to \calg_\qq$, where the second group is the group of rational isometric structures, defined as  follows.  An element in $\calg_\qq$ is represented by a pair $(Q, T)$ where $Q$ is a nonsingular symmetric bilinear form on an $n$--dimensional rational vector space for some $n$ and    $T$ is an isometry of that form.  (So, if $Q$ and $T$ are represented by matrices, $T^tQT=Q$.)  Such a pair is called {\it Witt trivial} if $Q$ vanishes on a $T$--invariant subspace of dimension $g$, where $2g = n$.  Isometric structures $(Q_1,T_1)$ and $(Q_2, T_2)$ are called {\it Witt equivalent} if   $(Q_1,T_1) \oplus  (-Q_2, T_2)$ is Witt trivial.  The set of equivalence classes forms the abelian group $\calg_\qq$ with operation induced by direct sum.
 
 The injective homomorphism $\psi\co \calg^\zz \to \calg_\qq$ is induced by the map $V \to (V+V^t, V^{-1}V^t)$.  To show this is well-defined, one proves that every class in $\calg^\zz$ can be represented by an invertible matrix.  It is clear that for an invertible Seifert matrix $V$, $\Delta_V(t) = \det(V) \Delta_{T}(t)$, where $T = V^{-1}V^t$ and  $\Delta_{T}(t)$ is the characteristic polynomial of $T$. 
 
 One  observations of Levine in~\cite{le2} is the following.
 
 \begin{theorem} For a symmetric monic irreducible polynomial $\delta$, the set of  Witt classes of isometric structures $(Q,T)$ for which $\Delta_T(t) = \delta^k$ for some $k$ is a subgroup $\calg_\qq^\delta \subset \calg_\qq$.   There is  an isomorphism  $\phi\co\calg_\qq \to  \oplus_\delta\  \calg_\qq^\delta$ given as the direct sum of projection maps $\phi_\delta$.  Here $\phi_\delta(Q,T)$ is the restriction of $(Q,T)$ to the subspace  annihilated by $\delta^k(T)$ for large $k$.
 
 \end{theorem}
 
 \noindent{\bf Notation.} We will denote $\phi_\delta(Q,T)$ by $(Q,T)^\delta$ and when we focus on the individual components, we will denote them $Q^\delta$ and $T^\delta$.
 \vskip.1in
 
 The next section will illustrate the explicit computation of the decomposition of an element in $\calg_\qq$.  First, we note the following  corollary.
 
 \begin{corollary}\label{maincor} If $(Q, T) \in \calg_\qq$ and $\phi_\delta(Q,T)$ is nontrivial, then $\delta(t) $ divides $\Delta_T(t)$.
 
 \end{corollary} 
 
 \begin{proof}If $\delta(t)$ does not divide $\Delta_T(t)$, then $\delta(T)$ acts as an isomorphism of the underlying vector space and thus has no kernel.
 \end{proof}


\section{Applications of rational invariants: examples}\label{exampsection}
In this section we determine the concordance genus of the eight and nine crossing knots that could not be resolved in~\cite{liv1}.

\begin{example}
{\boldmath {\bf If} $K = 8_{18}\co    g_3(K) = 3,\ g_c(K) = 3,\  g_4(K) =1$.}

The knot $8_{18}$ bounds a Seifert surface of genus 3 and has Alexander polynomial $\Delta_{8_{18}}(t)  = (t^2-t+1)^2(t^2-3t+1)$. Thus, $g_3(8_{18}) = 3$.

The unknotting number of $8_{18}$ is two, but the two crossing changes are of opposite signs.  Thus, $8_{18}$ bounds an immersed disk $D$ in $B^4$ with two double points of opposite sign.  Two small disks on $D$, one at each double point, can be removed and the pair of discs replaced by an annulus missing $D$, resulting in a embedded punctured torus in $B^4$ bounded by $8_{18}$.  Since the Alexander polynomial   has an irreducible symmetric factor with odd exponent, by the Fox-Milnor theorem $8_{18}$ is not slice, so $g_4(8_{18}) = 1$.

If $8_{18}$ is concordant to a knot $J$, then by the Fox-Milnor theorem $ t^2-3t+1$ divides $\Delta_J(t)$.  We wish to show that $ t^2-t+1$ also divides  $\Delta_J(t)$, which implies via the Fox-Milnor theorem that $(t^2-t+1)^2$ also divides $\Delta_J(t)$, implying that $\Delta_J(t)$ is of degree at least six, so that $g_3(J) \ge 3$, and hence $g_c(8_{18}) = 3.$  We note that $\sigma_{8_{18}}(\omega)$ is identically 0, so   signature calculations do not yield any information.

In order  to apply Corollary~\ref{maincor}, we need to show that for $8_{18}$ the projection of its isometric structure on $\calg_\qq^{t^2 - t+1}$ is nontrivial.  

The Seifert matrix for $8_{18}$, $  V_{8_{18}}$, is $6\times 6$.  For the associated isometric structure (defined on $\qq^6$), $(Q,T)$, we have that $\Delta_T(t) =  (t^2-t+1)^2(t^2-3t+1)$.  Thus,  $(T^2-T+1)^2(T^2-3T+1)$ annihilates all of $\qq^6$.  The summand of $\qq^6$ annihilated by a power of $T^2 - T +1$ is precisely the image of the transformation $  T^2-3T+1 $.  The transformation  $T$ can be expressed in matrix form by $V^{-1}V^t$ and a  basis for this image is simply a basis for  the column span of the matrix representation of  $  T^2-3T+1$, which can be found, for instance, using Gauss-Jordan elimination.  Here our calculations were aided by the computer program Maple.

Carrying out that calculation, it is found that the column span is 4--dimensional, with basis, say, $\{b_1, b_2, b_3, b_4\}$. A $4\times 4$ matrix representation of the quadratic form of $(Q,T)^{t^2-t-1}$ is given by  the matrix with entries $b_i ^t Q b_j$.  Starting with one particular Seifert matrix, as given in~\cite{lc}, the resulting matrix is:

$$ M = \left[ \begin {array}{cccc} 4&-2&0&-2\\\noalign{\medskip}-2&2&2&3
\\\noalign{\medskip}0&2&-2&1\\\noalign{\medskip}-2&3&1&2\end {array}
 \right] .$$

We claim this form is not trivial in the Witt group of symmetric bilinear forms over $\qq$, $W(\qq)$.  To see this, we apply a homomorphism $\partial_3\co W(\qq) \to W(\zz / 3 \zz)$, where $W(\zz / 3 \zz)$ is the Witt group of symmetric forms over the field with three elements.  The homomorphism $\partial_3$ can be defined via  the following algorithm.   First,     the matrix is diagonalized so that the diagonal entries are square free integers.  Those diagonal entries that are not divisible by $3$ are deleted, and those that are divisible by 3 are divided by $3$ and then reduced modulo 3.  We demonstrate this with the matrix $M$ above.  Details of the general theory of such homomorphisms can be found in~\cite[Chapter 4]{mh}.  In brief, there is a surjection $\partial \co W(\qq) \to \oplus_p W(\zz / p \zz)$ defined via such maps, and the kernel of $\partial$ is $W(\zz)$.

For the matrix $M$ above, when we diagonalize  we arrive at the  matrix:

$$M_1 = \left[ \begin {array}{cccc} 1&0&0&0\\\noalign{\medskip}0&1&0&0
\\\noalign{\medskip}0&0&-6&0\\\noalign{\medskip}0&0&0&-6\end {array}
 \right]. $$
 Removing the top two entries, which are not divisible by $3$,  dividing the last two entries by $3$, and reducing modulo $3$, gives the matrix with entries in $\zz / 3 \zz$
 
$$M_2 = \left[ \begin {array}{cc}1&0\\\noalign{\medskip}0&1\end {array}
 \right]. $$

This form is nontrivial in $W(\zz / 3 \zz)$ since the equation $x^2 + y^2 = 0$ does not have a nontrivial solution  in $\zz / 3 \zz$.
\end{example}


\begin{example}
{\boldmath {\bf If} $K = 9_{40}\co    g_3(K) = 3,\ g_c(K) = 3,\  g_4(K) =1$.}

This example is much like the previous one.  

We have that  $9_{40}$ bounds a Seifert surface of genus 3 and has Alexander polynomial $\Delta_{9_{40}}(t)  = (t^2-t+1)(t^2-3t+1)^2$. Thus, $g_3(9_{40}) = 3$.

As with  $8_{18}$, $9_{40}$ has unknotting number two, but the two crossing changes can be taken to have opposite signs, so    $g_4(9_{40}) = 1$.

If $9_{40}$ is concordant to a knot $J$, then by the Fox-Milnor theorem $ t^2-t+1$ divides $\Delta_J(t)$.  We wish to show that $ t^2-3t+1$ also divides  $\Delta_J(t)$, which implies via the Fox-Milnor theorem that $(t^2-3t+1)^2$ also divides $\Delta_J(t)$, implying that $\Delta_J(t)$ is of degree at least six, so that $g_3(J) \ge 3$, and hence $g_c(9_{40}) = 3.$  Unlike $8_{18}$,  $\sigma_{9_{40}}(\omega) = 2$, but this arises from the $t^2-t+1$ factor; the polynomial $t^2 -3t +1$ has no roots on the unit circle, so again signatures cannot be applied here.

To apply Corollary~\ref{maincor}, we will show that for $9_{40}$ the projection of its isometric structure on $\calg_\qq^{t^2 - 3t+1}$ is nontrivial.  

The calculation at this point is much as before.  The Seifert matrix for $9_{40}$,  $V_{9_{40}}$, is $6\times 6$.  For the associated isometric structure (defined on $\qq^6$), $(Q,T)$, we have that $\Delta_T(t) =  (t^2-t+1)(t^2-3t+1)^2$.  Thus,  $(T^2-T+1)^2(T^2-3T+1)$ annihilates all of $\qq^6$.  The summand of $\qq^6$ annihilated by a power of $T^2 - 3T +1$ is precisely the image of the transformation $  T^2-T+1 $.  Again,  $T$ can be expressed in matrix form by $V^{-1}V^t$ and a basis for the image of  $  T^2-T+1 $ is  a basis for  the column span of the matrix representation of  $ (T^2-3T+1)$.  Continuing with the calculation yields, as the matrix representing the bilinear form, the matrix 

$$ M =  \left[ \begin {array}{cccc} 2&-3&-1&-2\\\noalign{\medskip}-3&2&4&-2
\\\noalign{\medskip}-1&4&2&0\\\noalign{\medskip}-2&-2&0&-4\end {array}
 \right] 
 .$$

To see that this is not trivial in the Witt group of symmetric bilinear forms over $\qq$, $W(\qq)$, we apply a homomorphism $\partial_5\co W(\qq) \to W(\zz / 5 \zz)$, where $W(\zz / 5 \zz)$ is the Witt group of symmetric forms over the field with five elements.  The homomorphism $\partial_5$ can be defined in the same way as $\partial_3$:  diagonalize the matrix so that the diagonal entries are square free  integers; those diagonal entries that are not divisible by $5$ are deleted; and those that are divisible by $5$ are divided by $5$ and then reduced modulo $5$.  For the matrix $N$ above, upon diagonalizing we arrive at

$$M_1 =  \left[ \begin {array}{cccc} 2&0&0&0\\\noalign{\medskip}0&-10&0&0
\\\noalign{\medskip}0&0&1&0\\\noalign{\medskip}0&0&0&-5\end {array}
 \right] 
. $$
The first and third entries are not divisible by $5$, so are removed.  The remaining entries are divided by $5$ and reduced modulo 5 to yield: 
 
$$M_2 = \left[ \begin {array}{cc}3&0\\\noalign{\medskip}0&4\end {array}
 \right]. $$

This form is nontrivial in $W(\zz / 5 \zz)$ since the equation $3x^2 + 4y^2 = 0$ does not have a notrivial solution in $\zz / 5 \zz$.
\end{example}


\section{The knot $10_{82}$: algebraic concordance}

As described in Proposition 3 of the introduction, the situation with the knot $10_{82}$ is much more interesting, and delicate work using twisted Alexander polynomials as Casson-Gordon slicing obstructions is required.  In this section we show that $10_{82}$ is algebraically concordant to a knot $J$ with $g_3(J) = 2$.  In the next section we apply Casson-Gordon theory and twisted Alexander polynomials to show $g_c(10_{82}) = 4$.

The basic facts concerning the knot $10_{82}$ are as follows.  Its Alexander polynomial is $$\Delta_{10_{82}}(t) = \left( {t}^{4}-2\,{t}^{3}+{t}^{2}-2\,t+1 \right)  \left( {t}^{2}-t+1
 \right) ^{2}
.$$ Based on this and the fact that the knot bounds a Seifert surface of genus 4, we have $g_3(10_{82}) = 4$.    Since this knot has unknotting number 1 (and isn't slice by the Fox-Milnor theorem) we have $g_4(10_{82}) = 1$.  

We also have that $\sigma(10_{82}) = 2$.  This arises from a jump of the signature function at the unique root of   $  {t}^{4}-2\,{t}^{3}+{t}^{2}-2\,t+1   $   on the unit circle (with positive imaginary part).  The polynomial $t^2 - t+1$ also has a root on the unit circle, but the signature function for $10_{82}$ does not jump at that root.  

\begin{theorem}\label{thm10_82alg} In the direct sum decomposition $\calg_\qq \cong \oplus_\delta \calg_\qq ^\delta$, the image of the algebraic concordance class of $10_{82}$ in  $\calg_\qq ^{t^2-t+1}$ is Witt trivial and  $10_{82}$ has a 4--dimensional representative in $\calg_\qq$.
\end{theorem}  

\begin{proof}Proceeding as in the examples of the previous section, we can find a basis for the $t^2-t+1$ summand and compute a matrix representative of the bilinear form:

$$M = \left[ \begin {array}{cccc} -16&-26&12&-4\\\noalign{\medskip}-26&-40&
20&-6\\\noalign{\medskip}12&20&-12&2\\\noalign{\medskip}-4&-6&2&0
\end {array} \right].$$

Diagonilizing yields
$$Q^{t^2-t+1}  =  \left[ \begin {array}{cccc} -1&0&0&0\\\noalign{\medskip}0&1&0&0
\\\noalign{\medskip}0&0&-7&0\\\noalign{\medskip}0&0&0&7\end {array}
 \right]. $$
Clearly, this bilinear form is Witt trivial, but since we have lost track of the isometry, $T^{t^2-t+1}$,  it is not clear that the full isometric structure is Witt trivial.

To show this isometric structure $(Q,T)^{t^2-t+1}$ is Witt trivial, we apply the results of Levine~\cite{le2}.  In brief, the form will be Witt trivial if it is Witt trivial for all completions of $\qq$; that is, if it is Witt trivial in $\calg_\rr$ and $\calg_{\qq_p}$ for all $p$, where $\qq_p$ is the $p$--adic rationals.  In all cases, to check the triviality, one must further factor the Alexander polynomial over the field; fortunately,  the polynomial $t^2 -t+1$ is quadratic, so that if it becomes reducible in the field of interest, it factors into linear factors, and in that case, according to Levine, the isometric structure is automatically Witt trivial.  Thus, we can assume that $t^2-t+1$ is irreducible and no further decomposition of the isometric structure is required.

Next, another theorem in~\cite{le2}, based on a theorem of Milnor, states that if the characteristic polynomial of the isometry has even exponent (such as in our case, where the restriction has characteristic polynomial $(t^2 - t+1)^2$), then over the reals or $p$--adics, the isometric structure is trivial if and only if the associated bilinear form is Witt trivial.  Clearly, with the presence of the alternating signs in the diagonalization, the form, the extension of $Q^{t^2-t+1}$ over the completion $\qq_p$, will be Witt trivial  for all $p$.

\end{proof}

The previous theorem shows that   $(Q,T)^{t^4-2t^3+t^2-2t+1}$ is a $4$--dimensional rational representative of the algebraic concordance class of $10_{82}$.  We have a stronger result.

\begin{theorem}\label{intthm} The algebraic concordance class of $10_{82}$ has a $4 \times 4$ integral representative in $\calg^\zz$.

\end{theorem}  

\begin{proof} A general account of the structure of the integral, as opposed to rational, algebraic concordance group is contained in~\cite{st}.  The particular tools needed here are developed in~\cite{liv3}.

In the examples of Section~\ref{exampsection} we described how the bilinear form for each $\delta$--summand of an isometric structure is found.  That is how the matrix $Q^{t^2-t+1}$ was found in this section.  Once the basis is found for the summand of interest, one can find the matrix representation of the isometry, restricted to the summand, by usual linear algebra techniques: apply the transformation to each basis element and express the result in terms of the basis.  If this is done for the ${t}^{4}-2\,{t}^{3}+{t}^{2}-2\,t+1$ summand of the rational algebraic concordance class of $10_{82}$, one gets the following isometric structure.

$$(Q,T)^{t^4-2t^3+t^2-2t+1} = \left( \left[ \begin {array}{cccc} 0&2&0&0\\\noalign{\medskip}2&0&0&-2
\\\noalign{\medskip}0&0&-4&-2\\\noalign{\medskip}0&-2&-2&-8
\end {array} \right] ,\left[ \begin {array}{cccc} 1&1&-1&1\\\noalign{\medskip}1&0&0&-1
\\\noalign{\medskip}0&0&1&1\\\noalign{\medskip}-1&0&0&0\end {array}
 \right] \right).
 $$

According to Levine~\cite{le2} (except for a change of sign convention) a Seifert matrix $V$ with image an isometric structure $(Q,T)$ is given by $Q(1+T)^{-1}$, if it exists.  In our case, this yields: 

$$V = \left[ \begin {array}{cccc} 0&2&0&2\\\noalign{\medskip}0&0&0&-2
\\\noalign{\medskip}0&0&-2&0\\\noalign{\medskip}-2&0&-2&-4\end {array}
 \right].$$

 Notice that this is {\it not} an integral Seifert matrix, since $\det(V - V^t) \ne \pm 1$.  However, if we divide all the entries by $2$, the resulting matrix, $V_2$, is a Seifert matrix:
 
$$V_2 = \left[ \begin {array}{cccc} 0&1&0&1\\\noalign{\medskip}0&0&0&-1
\\\noalign{\medskip}0&0&-1&0\\\noalign{\medskip}-1&0&-1&-2\end {array}
 \right].$$ 
 
 Multiplication of a Seifert matrix by $2$ has the effect on the corresponding isometric structure of sending $(Q,T)$ to $ (2Q,T)$.  Since multiplication by a square ($2^2$) represents a change of basis and thus doesn't change the Witt class, multiplication by $2$ induces an involution of the Witt group $\calg_\ff$, for all $\ff$. We conclude the proof by observing that $V_2$ is fixed by this involution.  Equivalently, we show that $W = V \oplus -V_2$ is Witt trivial in $\calg^\qq$, Levine's rational analog of the integral algebraic concordance group. As usual, this will be done by mapping the class to   the Witt group of rational isometric structures.  This image class in $\calg_\qq$  is the direct sum   $$(V+V^t, V^{-1}V^t) \oplus -(V_2 + V_2^t , V_2^{-1}V_2^t).$$
 
 We have already seen that rationally the isometric structure for $10_{82}$, which we have been denoting $(Q,T)$, is Witt equivalent to $(V+V^t, V^{-1}V^t)$.  Thus, we want to show that 
 $$W = (Q, T) \oplus -(V_2 + V_2^t , V_2^{-1}V_2^t) $$ is Witt trivial. 
 
 The signature function is  identically 0,  so over the reals the form is trivial.  Thus, to apply Levine's theorem we need to consider only the $p$--adics.  Levine tells us to consider all primes $p$, but according to~\cite{liv3}, if $W$  is in the image of $\calg^\zz$ and is Witt trivial in $\calg_{\qq_p}$ for all prime divisors of  $\det(W) \text{disc}(\Delta_W(t))$, then $W$ is Witt trivial over $\qq$.  (Here {\it disc} denotes the polynomial discriminant.  See~\cite{liv3} or a standard algebra text such as~\cite{df} for details.)
 
 For the $W$ we are considering, the only primes that arise are $p=2$, $p=3$, and $p=7$.  (Discriminants are easily calculated using, for instance, Maple, or, for a quartic, by hand.)  Furthermore, we need to consider only the $t^4 - 2t^3+t^2-2t+1$ part of the 
 class, $W^{t^4 - 2t^3+t^2-2t+1}$, since we have already seen that $W^{t^2-t+1}$ will be Witt trivial for all $p$--adic completions.
 
 The prime $p=7$ is easily dispensed with.  Since $2$ is a square modulo $7$, it is also a  square over the $7$--adic rationals.  Thus, multiplying by $2$ does not change the Witt class of the associated isometric structure.
 
 For the prime $p=3$, the polynomial $\delta = t^4 - 2t^3+t^2-2t+1$ is irreducible in $\qq_3$ (it is irreducible modulo 3), and   the characteristic polynomial of the isometry in $W$ is $\delta^2$.  Thus, since $\delta$ has even exponent, as in the proof Theorem~\ref{thm10_82alg}   we only need to show that the bilinear form associated to $W$ is Witt trivial.  
The diagonalizations of $Q^{ t^4 - 2t^3+t^2-2t+1}$ and $(V_2 + V_2^t)$ are given by the  matrices 
 
 $$ \left[ \begin {array}{cccc} -7&0&0&0\\\noalign{\medskip}0&7&0&0
\\\noalign{\medskip}0&0&-14&0\\\noalign{\medskip}0&0&0&-2\end {array}
 \right]\ \text{and}\   \left[ \begin {array}{cccc} -14&0&0&0\\\noalign{\medskip}0&14&0&0
\\\noalign{\medskip}0&0&-7&0\\\noalign{\medskip}0&0&0&-1\end {array}
 \right]. $$
 Upon taking the direct sum of the first and the negative of the second, and removing the
elements that occur  with their negatives, we arrive at
$$ \left[ \begin {array}{cccc} -14&0&0&0\\\noalign{\medskip}0&-2&0&0
\\\noalign{\medskip}0&0&7&0\\\noalign{\medskip}0&0&0&1\end {array}
 \right]\ .$$

We next require a somewhat detailed analysis of relevant Witt classes over the $p$--adics.  One reference is~\cite{sch}, and~\cite{mh} gives much of the necessary background.   There is an isomorphism $\partial_3 \oplus \partial^e_3 \co W(\qq_3) \to W(\zz/3\zz) \oplus W(\zz/3\zz)$.  The first map is defined as the homomorphism $\partial_p$ used in Section~\ref{exampsection}.  Upon diagonalizing and making the entries square free, consider  only those  diagonal entries that are divisible by $3$, divide by 3, and reduce modulo 3. In the present case there are no such factors.  The map $\partial_3^e$ is defined similarly, except one considers only those factors that are not divisible by $3$, in our case all the entries.  The reduction modulo 3 is
 
  $$ \left[ \begin {array}{cccc} 1&0&0&0\\\noalign{\medskip}0&1&0&0
\\\noalign{\medskip}0&0&1&0\\\noalign{\medskip}0&0&0&1\end {array}
 \right]\ .$$
In $W(\zz/ 3\zz)$ this is Witt trivial.  A metabolizer is generated by the vectors $(1,0,1,1)$ and $(0,1,1,-1)$.  
 
  For the even prime $p = 2$ the maps $\partial_2$ and $\partial_2^e$ are not sufficient to determine the Witt class, so we use a direct argument. 
 To begin, we must first factor the characteristic polynomial over the $2$--adics.  However, one can check that ${t}^{4}-2\,{t}^{3}+{t}^{2}-2\,t+1$ is irreducible modulo $4$, and thus is irreducible in the $2$--adics.  As in the case of considering the prime $p = 3$, we are left to check the triviality of the following form (the same one as above) over the $2$--adic rationals, $\qq_2$.  
 
  $$ \left[ \begin {array}{cccc} -14&0&0&0\\\noalign{\medskip}0&-2&0&0
\\\noalign{\medskip}0&0&7&0\\\noalign{\medskip}0&0&0&1\end {array}
 \right]\ .$$

If $-7$ is a square in the $2$-adics, that is, if there is a $2$--adic $a$ such that $a^2 = -7$, this form will be Witt trivial; a metabolizer would be spanned by the vectors $(1,a,0,0)$ and $(0,0,1,a)$.    A rational integer    is a square  in $\qq_2$ if and only if it is of the form $2^n u$, where $n$ is even $u$ is an odd integer congruent to 1 modulo 8 (see for instance~\cite{se}).  Since $-7$ is congruent to $1$ modulo $8$, we are done.
 \end{proof}

\begin{remark} A more ad hoc approach to showing that there is a $4\times 4$ integral representative of the algebraic concordance class of  $10_{82}$ consists of finding a particular  genus 2 knot which is algebraically concordant to $10_{82}$.  A computer search reveals that $-9_{42}$ is one such knot.  To prove this, one needs to show that $J = 10_{82} \# 9_{42}$ is algebraically slice.  

Both knots have the same signature function,  and hence  $J = 0 \in \calg_\rr$.  Also, the Alexander polynomial of $J$ is $\Delta_J(t) = ( t^4 - 2t^3+t^2-2t +1)^2(t^2-t+1)^2$.  The image of $J$ in $\calg_\qq^{t^2-t+1}$ is, as seen earlier, trivial.  Thus, we need to show that the image in $\calg_\qq^{t^4 - 2t^3+t^2-2t +1}$ is trivial.   As mentioned earlier, according to Levine one now needs to check triviality in all $p$--adic completions of $\qq$, but as in the proof of Theorem~\ref{intthm},   one need   check only at  the primes $p = 2$, $p=3$, and $p=7$.   The actual calculations are much the same as in the proof of Theorem~\ref{intthm} and thus won't be repeated.  However, we should comment on one aspect of the argument.  In defining the isometric structure of a knot we needed to work with a nonsingular Seifert matrix;  that is only required to define the isometric structure.  Thus one can work with the $6 \times 6$ Seifert matrix of $9_{42}$ given in~\cite{lc} even though that matrix has  determinant 0.  The details are not included here.
\end{remark}

\section{Casson-Gordon invariants, twisted polynomials, and ribbon obstructions}

This section is devoted to the proof of the following result, Proposition 3 of the introduction. Recall that $g_3(10_{82}) = 4$ and $g_4(10_{82}) = 1$.  Also  the Fox-Milnor theorem applies to show $2 \le g_c(10_{82}) \le 4$.   
 
\begin{theorem} \label{slicethm}  There does not exist a knot $J$ with $g_3(J) \le 3$ such that $10_{82} \ \# \ -J$ is slice.  In  particular, $g_c(10_{82}) = 4$.   
\end{theorem}


\subsection{Casson-Gordon invariants}

Let $K$ by a knot with $p$--fold branched cover $M_p$.  For simplicity we assume that $p$ is a prime.  Let $\chi\co H_1(M_p) \to \zz/q\zz$, where $q$ is a prime.  In this setting there is defined in~\cite{cg1} a Casson-Gordon invariant:  $$\tau(K,\chi) \in W(\qq(\zeta_q)) \otimes \zz[\frac{1}{q}],$$ where $\zeta_q$ is a primitive $q$--root of unity.  The main theorem of~\cite{cg1} states:

\begin{theorem}\label{cgthm} If $K$ is slice, there exists a metabolizer $M$ for the linking form of $H_1(M_p)$ such that for all $\chi$ that vanish on $M$, $\tau(K,\chi) =  0 \in W(\qq(\zeta_{q^r})) \otimes \zz[\frac{1}{q}],$ for some $r \ge 1$.
\end{theorem}

\noindent{\bf Comments.}  A {\it metabolizer} $M \subset H_1(M_p)$ is a subgroup for which 
$|M|^2 = |H_1(M_p)|$ and on which the linking form of $H_1(M_p)$ is identically $0$.  \vskip.1in

\noindent{\bf Addendum.} In the case that $q$ is odd or that $K$ is ribbon, it follows from the work of~\cite{cg1} that one can let $r = 1$ in the statement of Theorem~\ref{cgthm}. In our case we need to work with $q = 2$.  This introduces algebraic difficulties that have not appeared in past work using Casson-Gordon theory.  These difficulties seem unavoidable in working with $10_{82}$, and necessitate a Galois theoryargument in the next section.

\subsection{Twisted Alexander polynomials}

We now summarize the results of~\cite{hkl, kl1}.  Given the pair $(M_p, \chi)$, one lets $\overline{M}_p$ be the $3$--manifold that is the $p$--fold cyclic cover of 0--surgery on $K$.  There is an induced character $\overline{\chi} \co H_1(\overline{M}_p) \to \zz/p\zz \oplus \zz$.  One then has that $\qq(\zeta_q)[t,t^{-1}]$ is a module over the group ring $\zz[\pi_1(\overline{M}_p)]$ and we can consider the twisted homology group $H_1( \overline{M}_p, \zz[\pi_1(\overline{M}_p])$ as a  $\qq(\zeta_q)[t,t^{-1}]$--module. This will be a torsion module, and we  have the following.

\begin{definition} With the notation above, $\Delta_{K,\chi}(t) \in   \qq(\zeta_q)[t,t^{-1}]$ is the order of the module $H_1( \overline{M}_p, \zz[\pi_1(\overline{M}_p)])$.  It is well-defined up to multiplication by units, that is, by elements of the form $at^i$.
\end{definition}

In~\cite{kl1} it is proved that for $q$ odd, $\Delta_{K,\chi}(t)$ is, roughly, the discriminant of the Casson-Gordon invariant.  From that one can conclude that if $K$ is slice, then for appropriate $\chi$, $\Delta_{K,\chi}(t)$ will factor as $f(t) \overline{f(t^{-1})}$, $f \in   \qq(\zeta_q)[t,t^{-1}]$,    (with perhaps an additional factor of $(1-t)$ appearing).  However, in~\cite{kl1} a direct proof of this factoring condition is given, and that proof does not make use of the condition that $q$ is odd in attaining a factoring condition.  However, as in the original work of~\cite{cg1}, if $q=2$ then the factoring can only be shown to be over the field $\qq(\zeta_{q^r})$ for some $r$.  In summary we have:

\begin{theorem} \label{cg1thm} If $K$ is slice, there exists a metabolizer $M$ for the linking form of $H_1(M_p)$ such that for all $\chi$ that vanish on $M$, $\Delta(K,\chi) =  at^i f(t) \overline{f(t^{-1})}(1-t)^s$ for some $a \in \qq(\zeta_{q^r}), i \in \zz$ and $f \in   \qq(\zeta_{q^r})[t,t^{-1}]$, where $r $ is some positive integer.  For $\chi$ nontrivial, $s = 1$, and for $\chi$ trivial, $s = 0$.

\end{theorem}

\subsection{Homology of Covers} 
To apply Theorem~\ref{cg1thm} we need to understand the metabolizers of the relevant branched covers of the knots of interest.

Suppose that $10_{82}$ is concordant to a knot $J$ with $g_c(J) \le 3$.  Then by the Fox-Milnor theorem we have that for some integer $a$, 

$$\Delta_J(t) = (t^4-2t^3+t^2-2t +1)(at-(a-1))((a-1)t-a).$$ 

 \begin{theorem}  $ \ $ 
 \begin{enumerate}
  \item The homology of the 3--fold branched cover of $S^3$ over $10_{82}$  is given by  $H_1(M_3(10_{82})) = \zz / 8 \zz\ \oplus\ \zz /8 \zz$.  
  \item The homology of the 3--fold branched cover of $S^3$ over $J$ satisfies  $H_1(M_3(J)) = \zz / 2 \zz\ \oplus\ \zz / 2 \zz\ \oplus\ T$, where the order of $T$ is odd.  
  \item For each metabolizer   $M \subset H_1(M_3(10_{84}   \#  -J)) $ there is a nontrivial character  $\chi \co  H_1(M_3(10_{84}   \#   -J)) \to \zz / 2 \zz$ which vanishes on $M$ and also vanishes on $ H_1(M_3(  -J))$.
\end{enumerate}
 \end{theorem}
 
 \begin{proof} $ \ $
 \begin{enumerate}
 \item This is a standard calculation in knot theory; see for instance~\cite{ro}.  See also the proof of (2) next.
 
 \item  A theorem of Fox~\cite{fo1}  states that the order of the homology of the 3--fold branched cover of a knot $K$ is given by $|\Delta(\zeta_3)\Delta(\overline{\zeta_3})|.$  From this a direct calculation based on our given form of $\Delta_J(t)$ yields $|H_1(M_3(J)) | = 4(a^3 - (a-1)^3)^2$.  The theorem of Plans~\cite{pl} (see also~\cite{ro}) states that for odd prime powers $p$, $H_1(M_p(K))$ is always of the form $T_1 \ \oplus \ T_1$ for some torsion group $T_1$.  Since $a^3 - (a-1)^3 $ is odd, the result follows.

 \item The 2--torsion in $H_1(M_3(10_{84} \ \# \ -J)) $ is $H_2 \cong \zz/ 8\zz\ \oplus\ \zz/ 8\zz\ \oplus\ \zz/ 2\zz\ \oplus\ \zz/ 2\zz$.  We let $M_2$ be the 2--torsion in $M$:  $| M_2| = 16$.   For notation, we use the coordinates given by the direct sum decomposition of $H_2$; abbreviate $v_1 = (1,0,0,0)$, $v_2 = (0,1,0,0)$, $v_3 = (0,0,1,0)$, and $v_4 = (0,0,0,1)$.  A generating set for $M_2$ can be simplified using a Gauss-Jordan procedure to be of the form 
 $$\{   (a_1,b_1,c_1,d_1),   (0, b_2,c_2,d_2),   (0,0,c_3,d_3), (0,0,0,d_4)\}.$$

 If $a_1$ is even, then the character $\chi\co H_2 \to \zz/ 2\zz$ that takes value 1 on   $v_1$  and 0 on all other $v_i$ vanishes on $M_2$ and  $ H_1(M_3(  -J))$, as desired.
 
 If $a_1$ is odd, then $b_2$ must be even, or else the first two generators, $ (a_1,b_1,c_1,d_1)$ and $  (0, b_2,c_2,d_2)$,  would generate a
 subgroup of order 64.  If $b_1$ is also even, then we can let $\chi$ be the character that takes value 1 on $v_2$ and 0 on all other $v_i$.  If $b_1$ is odd, we can let $\chi$ be the character that takes value 1 on   $v_1$ and $v_2$, and takes value 0 on $v_3$ and $v_4$.  In either case, $\chi$ will have the desired properties.
  
 \end{enumerate}
 \end{proof}
 We now wish to compute the twisted polynomial for the $\chi$ given in the previous theorem.  We state the outcome as the following lemma.
 
 \begin{lemma} \label{charlem}
  For the character $\chi \co H_1(M_3(10_{84} \ \# \ -J))\to \zz / 2\zz$ given  above,  $ \Delta_{10_{82}  \#   J, \chi}(t) = (t^4 - 8t^3 +10t^2 - 8t +1)g(t)g(t^{-1})$ for some $g(t)$.
 
\end{lemma}

 \begin{proof}
 
  By   multiplicativity, the  twisted polynomial  is given by
 $\Delta_{10_{82}, \chi_1}(t) \Delta_{J, \chi_2}(t),  $
   where $\chi_1$ and $\chi_2$ are the two restrictions. Notice that $\chi_2$ is trivial.

Since $J$ is unknown, computing its general twisted polynomials would be impossible. In the present case however,  $\chi_2$ is trivial and the twisted polynomial is determined by the Alexander polynomial of $J$.  The simplest formulation, given in~\cite{hkl}, is as follows.  For any polynomial $f$ and prime $p$  we can form the product, 
  $$N_p(f)(t) =  \prod_{i=0}^{p-1}   \Delta_J({\zeta_p}^i x)|_{x^p= t}.$$
  (The product on the right will be a polynomial in $x^p$, so the substitution $x^p = t$ does yield a polynomial in $t$.)  Then, according to~\cite{hkl}, for the trivial $\chi$, $\Delta_{K, \chi}(t) = N_p(\Delta_K(t))$, where $p$ is the  degree of the covering space; that is,  if we are working with $H_1(M_p)$.
  
Since we are interested in whether the twisted polynomial factors as $f(t)\overline{f(t^{-1})}$ we can ignore the pair of factors of $\Delta_J(t)$ that are already of this form and compute 
$N_2(t^4-2t^3+t^2-2t+1)$.  The calculation yields:
$$\Delta_{J, \chi_2}(t) = (t^4 - 8t^3 +10t^2 - 8t +1) g(t)g(t^{-1}).$$
(Notice that since we are working with a character to $\zz_2$, the polynomial will be real; the primitive square root of 1 is $-1$.)

For $10_{82}$ a calculation based on the results of~\cite{hkl} yields $$\Delta_{10_{82}, \chi_1}(t) = (t^2 + 2t-1)(t^2-2t-1) (t-1)^2.$$
Note, this is of the form  $a t^k h(t)h(t^{-1})$.  Thus, the statement of the lemma follows.

 \end{proof}
 
    As a consequence, we can now complete the proof of Theorem~\ref{slicethm}.
    
    \begin{proof}[Proof, Theorem~\ref{slicethm}]  
      If  $10_{82}\ \# \ J$ 
   is slice for some $J$ with $g_3(J) \le 3$, then by Lemma~\ref{charlem} the polynomial $\Delta(t) = t^4 - 8t^3 +10t^2 - 8t +1$ would
   factor as $t^2 g(t)g(t^{-1})$ in $\qq(\zeta_{2^r})[t]$ for some $r$. 
In fact, over $\qq(\zeta_{8})$, $\Delta(t)$ does factor into two irreducible symmetric factors:
$$  \Delta(t)  = (t^2 + (2\zeta_{8} -2\zeta_{8}^3 -4)t +1)(t^2  +   (-2\zeta_{8} +2\zeta_{8}^3 -4)t +1).
$$
These factors however are not complex conjugates; the coefficients are all real.

Thus, if $\Delta(t)$ does factor as $t^2 g(t)g(t^{-1})$ for some $g(t) \in   \qq(\zeta_{2^r})$ and $r\ge 1$, then  $\Delta(t)$ would factor into linear factors, so that $g(t)$ would be the product of one linear factor of each of the two irreducible quadratic factors of $\Delta(t)$ in  $\qq(\zeta_{8})[t]$.

In particular, we would have that the splitting field $\ff$ for $\Delta(t)$ would be a subfield of  $\qq(\zeta_{2^r})$.  The Galois group of the splitting field for $\Delta(t)$ is the nonabelian dihedral group with eight elements, as can be computed by Maple.  On the other hand, this Galois group  should be a quotient of the Galois group of the extension  $\qq(\zeta_{2^r})$, which is abelian.  This gives the desired contradiction.

      \end{proof}

 \section{Problems}
 
 \subsection{Smooth invariants}  The distinction between the smooth and topological locally flat category, with respect to the study of concordance genus, is  made clear by the following problem:  working in the smooth category, find a knot $K$ with $\Delta_K(t) = 1$, and for which $g_c(K) \ne g_4(K)$.  Although a host of tools are now available that distinguish smooth and topological concordance (for instance, based  on gauge theory~\cite{cog}, Heegaard-Floer homology~\cite{os}, and Khovanov homology~\cite{ra}), it is not clear that any of these can be applied to this problem.  
 
 \subsection{Topological obstructions}  The results of this paper, including those using Casson-Gordon invariants, apply in the topological category.  In~\cite{liv1} Casson-Gordon invariants were used to analyze the concordance genus for algebraically slice knots.  There should be examples for which Casson-Gordon methods fail, but for which the techniques of~\cite{cot} do apply.  
 
To make the issue concrete,  here is a specific problem:  For every $n \in \frac{1}{2}\zz$, find a knot $K$ with $g_4(K) = 1$ such that there exist knots $J$ with $g_3(J) = 1$, and $K \ \# \ -J \in \calf_{n}$, but for all such $J$, $K \ \# \ -J \notin \calf_{n+  .5}$.  Here $\{\calf_{n}\}$ represents to filtration of $\calc$ defined in~\cite{cot}.
 
 \subsection{Concordance relations and torsion} The work in this paper is closely related to the problem of finding, and obstructing, concordance relations between low-crossing number knots.  For instance, we have seen that $10_{82}\ \# \ 9_{42}$ is not slice, but is algebraically slice.  It remains possible that this knot represents torsion in the concordance group; that is,   $k 10_{82} + k  9_{42} = 0 \in \calc$ for some $k$.  Many new methods have been applied to obstruct torsion in $\calc$, and these have   resolved many of the basic examples taken from the table of low-crossing number knots.  See, for example~\cite{grs, jn,  lis, ln1}.  New test cases can be found by examining such algebraic concordance relations.
 
 \subsection{(Sub)multiplicative properties of $g_4$ and $g_c$.}  It is clear that for all $K$ and $n\ge 0$,  $g_4(nK) \le ng_4(K)$ and $g_c(nK) \le ng_c(K)$.  Knots that represent torsion in $\calc$ can be used to build a variety of  examples demonstrating  that the inequality can be strict.  For instance,  $g_4(n (  3_1 \ \# \ 4_1)) = g_c(n (  3_1 \ \# \ 4_1) )= n + \epsilon$, where $\epsilon = 0$ if $n$ is even and $\epsilon = 1$ if $n$ is odd.
 
Interesting results can be observed by considering, for a fixed knot $K$, the quotient $g_4(nK)/n$ for $n$ large.  For knots that represent torsion in $\calc$ there is a limiting value: $\lim_{n \to \infty}   g_4(nK)/n = 0$.  If the 4--genus of a knot is determined by it classical signature  (that is, $\sigma(K) = 2g_4(K)$) then  again there is a limiting value: we have $\lim_{n \to \infty}   g_4(nK)/n = g_4(K)$.  This applies for the trefoil knot.  

There are more interesting examples.  For instance,  for the knot $K = 8_1$ one has that $g_4(K) = g_4(2K) =1$ and $\limsup_n  g_4(nK)/n \le 1/2$.  In fact a limit exists, but it is unknown whether for this knot, or any knot, the limiting value can be a non-integer.

The same questions can be asked regarding $g_c$, but here there are few tools to employ beyond basic signature and Alexander polynomial  methods.  For instance, for the knots  $8_{18}, 9_{40}$ and $10_{82}$, the methods of this paper do not distinguish the limiting behavior of $g_c$ and $g_4$.  
    


\end{document}